\documentclass[12pt,a4paper]{amsart}

\usepackage[utf8]{inputenc}
\usepackage[T1]{fontenc}
\usepackage[english]{babel}
\usepackage{cite}

\usepackage{indentfirst}
\usepackage{amssymb}
\usepackage{amsfonts}
\usepackage{amsmath}
\usepackage{amsthm}
\usepackage[top=2.5cm, bottom=2.5cm, left=2.5cm, right=2.5cm]{geometry}
\usepackage{amsopn}
\usepackage[colorlinks]{hyperref}

\theoremstyle{plain}
\newtheorem{thm}{Theorem}

\newtheorem{lem}[thm]{Lemma}
\newtheorem{prop}[thm]{Proposition}
\newtheorem{cor}[thm]{Corollary}

\theoremstyle{definition}

\newtheorem*{example*}{Example}

\newtheorem*{rem*}{Remark}
\newtheorem{rem}[thm]{Remark}

\newcommand{\al}{\alpha}
\newcommand{\be}{\beta}

\newcommand{\R}{\mathbb{R}}

\DeclareMathOperator{\sgn}{sgn}
\DeclareMathOperator{\dist}{dist}

\DeclareSymbolFont{bbsymbol}{U}{bbold}{m}{n}
\DeclareMathSymbol{\ind}{\mathbin}{bbsymbol}{'061}
\title[Sharp fractional Hardy inequalities with a~remainder for $1<p<2$]{Sharp fractional Hardy inequalities with a~remainder for $\boldsymbol{1<p<2}$}
\author[B{.} Dyda]{Bart{\l}omiej Dyda}
\author[M{.} Kijaczko]{Micha\l{} Kijaczko}

\address[B.D. and M.K.]{Faculty of Pure and Applied Mathematics\\ Wroc{\l}aw University 
	of Science and Technology\\
	Wybrze\.ze Wyspia\'nskiego 27,
	50-370 Wroc{\l}aw, Poland
}
\email{bdyda@pwr.edu.pl\qquad dyda@math.uni-bielefeld.de}
\email{michal.kijaczko@pwr.edu.pl}
\begin{document}

\begin{abstract}
The main purpose of this article is to obtain (weighted) fractional Hardy inequalities with a remainder and fractional Hardy--Sobolev--Maz'ya inequalities valid for $1<p<2$.
\end{abstract}

\maketitle

\section{Introduction}
In this paper we are interested in (weighted) fractional order Hardy inequalities. The basic form of such inequality is
\begin{equation}\label{generalweightedfractionalHardy}
\int_{\Omega}\int_{\Omega}\frac{|u(x)-u(y)|^{p}}{|x-y|^{d+sp}}\dist(x,\partial\Omega)^{-\alpha}\dist(y,\partial\Omega)^{-\beta}\,dy\,dx\geq C\int_{\Omega}\frac{|u(x)|^{p}}{\dist(x,\partial\Omega)^{sp+\alpha+\beta}}\,dx, 
\end{equation}
where $u\in C_{c}^{1}(\Omega)$ ( continuously differentiable functions with compact support), $C=C(d,s,p,\al,\be,\Omega)$ is a universal constant independent of $u$, $\Omega$ is a nonempty, proper, open subset of $\R^d$, $\dist(x,\partial\Omega)=\inf_{y\in\partial\Omega}|x-y|$ denotes the distance to the boundary and $s$, $p$, $\al$ and $\beta$ are some parameters.

Let us suppose that the constant $C$ is the largest possible for which \eqref{generalweightedfractionalHardy} holds. Such $C$ has been explicitly found for $\Omega$ being either $\R^d\setminus \{0\}$, the halfspace $\R^d_+$,
  or a~convex domain, see \cite{MR2469027, MR2663757, MR2723817, LossSloane, sharpweighted}.
In the case when $p\geq 2$, Frank and Seiringer showed, in a~more general setting,
that one can add an additional term  on the right hand side of inequality \eqref{generalweightedfractionalHardy}, so called \emph{remainder},
see \cite{MR2469027, MR2723817} . The resulting inequality  has been applied in \cite{MR2823046} to obtain Hardy--Sobolev--Maz'ya inequality.

In this note, our main  goal is to extend these results to the case of $1<p<2$,
  see Proposition~\ref{remainder1<p<2}, which  holds in a~setting of more general Hardy inequalities.
  Analogously, it also yields Hardy--Sobolev--Maz'ya inequalities,
see Theorems~\ref{tw12} and~\ref{twHSMgeneraldomains}.

For the sake of completeness let us mention that weighted fractional Gagliardo seminorms appearing on the left hand side of \eqref{generalweightedfractionalHardy} and its modifications were investigated recently by many authors -- for example in \cite{MR3420496} and \cite{MK} in connection with weighted fractional Sobolev spaces, or in \cite{ID} as an object of study in the theory of interpolation spaces. The weighted fractional Hardy inequality \eqref{generalweightedfractionalHardy} appeared in \cite{MR362603f1} for $\Omega=\R^d\setminus\{0\}$ or in \cite{MR3803664,MR3237044} for more general domains.

\subsection{Non-linear ground state representation of Frank and Seiringer}\label{subs:FS}
Following~\cite{MR2469027}, let $k(x,y)$ be a symmetric, positive kernel and $\Omega\subset\R^d$ be an open, nonempty set. Let us define the functional 
$$
E[u]=\int_{\Omega}\int_{\Omega}|u(x)-u(y)|^{p}k(x,y)\,dy\,dx
$$
and 
$$
V_{\varepsilon}(x)=2w(x)^{-p+1}\int_\Omega\left(w(x)-w(y)\right)\left|w(x)-w(y)\right|^{p-2}k_{\varepsilon}(x,y)\,dy,
$$
where $w$ is a positive, measurable function on $\Omega$ and $\{k_{\varepsilon}(x,y)\}_{\varepsilon>0}$ is a family of measurable, symmetric kernels satisfying the assumptions $0\leq k_\varepsilon(x,y)\leq k(x,y)$, $\lim_{\varepsilon\rightarrow 0}k_\varepsilon(x,y)=k(x,y)$ for all $x,y\in\Omega$. Then, if the integrals defining $V_\varepsilon(x)$ are absolutely convergent for almost all $x\in\Omega$ and converge weakly to some $V$ in $L^{1}_{loc}(\Omega)$, as $\varepsilon$ tends to $0$, we have the following Hardy-type inequality, 
\begin{equation}\label{generalhardy}
 E[u]\geq\int_{\Omega}|u(x)|^{p}V(x)\,dx,   
\end{equation}
for all compactly supported $u$ with $\int_{\Omega}|u(x)|^{p}V_{+}(x)\,dx$ finite, see \cite[Proposition 2.2]{MR2469027}.

Moreover, if in addition $p\geq 2$, then the inequality \eqref{generalhardy} can be improved by a remainder
\begin{equation}
E_{w}[v]=\int_{\Omega}\int_{\Omega}|v(x)-v(y)|^{p}w(x)^{\frac{p}{2}}k(x,y)w(y)^{\frac{p}{2}}\,dy\,dx,\quad u=wv,
\end{equation}
that is the inequality 
\begin{equation}\label{generalhardyremainder}
 E[u]-\int_{\Omega}|u(x)|^{p}V(x)\,dx\geq c_{p}E_{w}[v]  
\end{equation}
holds with the same assumptions as in \eqref{generalhardy}, with the constant $c_p$ given by
\begin{equation}\label{cp}
c_{p}=\min_{0<\tau<\frac{1}{2}}\left((1-\tau)^{p}-\tau^{p}+p\tau^{p-1}\right),  
\end{equation}
see \cite[Proposition 2.3]{MR2469027}. When $p=2$, the inequality \eqref{generalhardyremainder} becomes an equality.

\subsection{Remainder for $\textbf{1<\emph{p}<2}$}
We propose an extension of \eqref{generalhardyremainder} to the case when $1<p<2$.
In what follows we use the notation of the so called \emph{French power}, that is for  $a\in\R$ and $k>0$ we define
$$
a^{\langle k \rangle }=|a|^{k}\sgn(a).
$$
It holds $\frac{d}{da} a^{\langle k \rangle }=k|a|^{k-1}$.

For $1<p<2$ we denote
\[
\widetilde{E}_w[u]:=\int_{\Omega}\int_{\Omega}\left(u(x)^{\langle p/2 \rangle }-u(y)^{\langle p/2 \rangle }\right)^2W(x,y)k(x,y)\,dy\,dx,
\]
where 
\[
W(x,y):=\min\{w(x),w(y)\}\max\{w(x),w(y)\}^{p-1}=w(x)w(y)\max\{w(x),w(y)\}^{p-2}.
\]
The functional defined above will serve as a remainder in a following general fractional Hardy-type inequality for $1<p<2$.
\begin{prop}\label{remainder1<p<2}
  Let $u$, $w$ and $V$ be like in Subsection~\ref{subs:FS}, in particular, we assume that $u\in C_c^1(\Omega)$ and $\int_{\Omega}|u(x)|^{p}V_{+}(x)\,dx < \infty$. We also assume that $E[u]<\infty$.
  Let $v=u/w$ and $1<p<2$.
  For real-valued functions $u$ it holds
\begin{equation}\label{generalhardyremainder1<p<21}
 E[u]-\int_{\Omega}|u(x)|^{p}V(x)\,dx\geq C_p\widetilde{E}_{w}[v],
\end{equation}
with $C_p$ defined in \eqref{Cp}.
When $u\geq 0$, we have
\begin{equation}\label{generalhardyremainder1<p<22}
 E[u]-\int_{\Omega}|u(x)|^{p}V(x)\,dx\geq (p-1)\widetilde{E}_{w}[v].      
\end{equation}
Finally,  for all complex-valued functions $u$,
\begin{equation}\label{generalhardyremainder1<p<23}
 E[u]-\int_{\Omega}|u(x)|^{p}V(x)\,dx\geq (p-1)\widetilde{E}_{w}[|v|].      
\end{equation}
\end{prop}
  Note that for real-valued functions, both inequalities \eqref{generalhardyremainder1<p<21} and
  \eqref{generalhardyremainder1<p<23} hold. Since $C_p < p-1$, 
  the constant is better in the latter inequality, for the price of taking absolute value of $v$
  in $\widetilde{E}_{w}[|v|]$.

As a consequence of Proposition \ref{remainder1<p<2} we obtain weighted fractional Hardy inequalities with a remainder for $\R^d$ and $\R^d_+$ for $1<p<2$.

\begin{thm}[\textbf{Sharp weighted fractional Hardy inequality for 1<$\bold{\emph{p}}$<2 with a remainder on $\R^d$}]\label{tw10}
Let $0<s<1,\,1<p<2$, $\al,\be,\al+\be\in (-sp,d)$. Then for all $u\in C_{c}^1(\R^d)$, when $sp+\al+\be<d$ and for all $u\in C_{c}^1(\R^d\setminus\{0\})$, when $sp+\al+\be>d$, the following inequality holds,
\begin{align}
&\int_{\R^d}\int_{\R^d}\frac{|u(x)-u(y)|^{p}}{|x-y|^{d+sp}}|x|^{-\al}|y|^{-\be}\,dy\,dx-\mathcal{C}\int_{\R^d}\frac{|u(x)|^{p}}{|x|^{sp+\al+\be}}\,dx \label{HardyRd} \\
&\geq C_p\int_{\R^d}\int_{\R^d}\frac{\left(v(x)^{\langle p/2 \rangle}-v(y)^{\langle p/2 \rangle}\right)^{2}}{|x-y|^{d+sp}}W(x,y)|x|^{-\al}|y|^{-\be}\,dy\,dx, \nonumber
\end{align}
where $v(x)=u(x)|x|^{\frac{d-\al-\be-sp}{p}}$ if $u$ is real-valued,
and $v(x)=|u(x)|\,|x|^{\frac{d-\al-\be-sp}{p}}$ if $u$ is complex-valued,
$$
W(x,y)=\min\left\{|x|^{-\frac{d-\al-\be-sp}{p}},|y|^{-\frac{d-\al-\be-sp}{p}}\right\}\max\left\{|x|^{-\frac{d-\al-\be-sp}{p}},|y|^{-\frac{d-\al-\be-sp}{p}}\right\}^{p-1},
$$
and the constants $C_p$ and $\mathcal{C}>0$ are given by \eqref{Cp} and \eqref{C}. When $u\geq 0$ or $u$ is complex-valued, the constant $C_p$ in the inequality above can be replaced by $p-1$.
\end{thm}

\begin{thm}[\textbf{Sharp weighted fractional Hardy inequality for $\boldsymbol{1<p<2}$ with a remainder on $\R^d_+$}]\label{tw11}
Let $0<s<1,\,1<p<2$, $\al,\be,\al+\be\in (-1,sp)$, $\al+\be+sp\neq 1$. 
Then for all $u\in C_c^1(\R^d_+)$, when $\al+\be+sp>1$ and $u\in C_c^1(\overline{\R^d_+})$, when $\al+\be+sp<1$, the following inequality holds,
\begin{align}
&\int_{\R^d_+}\int_{\R^d_+}\frac{|u(x)-u(y)|^{p}}{|x-y|^{d+sp}}x_d^{\al}\,y_d^{\be}\,dy\,dx-\mathcal{D}\int_{\R^d_+}\frac{|u(x)|^{p}}{x_d^{sp-\al-\be}}\,dx \label{HardyHS} \\
&\geq C_p\int_{\R^d_+}\int_{\R^d_+}\frac{\left(v(x)^{\langle p/2 \rangle}-v(y)^{\langle p/2 \rangle}\right)^{2}}{|x-y|^{d+sp}}V(x,y)x_d^{\al}y_d^{\be}\,dy\,dx, \nonumber
\end{align}
where $v(x)=u(x)x_d^{\frac{1+\al+\be-sp}{p}}$  if $u$ is real-valued,
and $v(x)=|u(x)|\, x_d^{\frac{1+\al+\be-sp}{p}}$  if $u$ is complex-valued,
$$
V(x,y)=\min\left\{x_d^{-\frac{1+\al+\be-sp}{p}},y_d^{-\frac{1+\al+\be-sp}{p}}\right\}\max\left\{x_d^{-\frac{1+\al+\be-sp}{p}},y_d^{-\frac{1+\al+\be-sp}{p}}\right\}^{p-1},
$$
and the constants $C_p$ and $\mathcal{D}>0$ are given by \eqref{Cp} and \eqref{D}. When $u\geq 0$ or $u$ is complex-valued, the constant $C_p$ in the inequality above can be replaced by $p-1$.
\end{thm}

\subsection{Application: fractional Hardy--Sobolev--Maz'ya inequalities}

The  obtained inequalities with the remainder may be used to prove that the (weighted) fractional Hardy--Sobolev--Maz'ya inequality on $\R^d_+$ and (unweighted) fractional Hardy--Sobolev--Maz'ya inequality on general domains are valid also in the range $1<p<2$. This seems to be new even for the unweighted case $\alpha=\beta=0$. Hardy--Sobolev--Maz'ya inequalities are interesting of its own, because they combine both Hardy and Sobolev inequalities. For the information about the unweighted local and nonlocal Hardy--Sobolev--Maz'ya inequality we refer the Reader to \cite{MR2863763, MR2823046, MR2424899, MR2910984}, and to \cite[Theorem 3]{sharpweighted} for the weighted form of this inequality, valid for $p\geq 2$.

We make  our statement precise below.

\begin{thm}[\textbf{Weighted fractional Hardy--Sobolev--Maz'ya inequality on $\R^d_+$ for $\boldsymbol{1<p<2}$}]\label{tw12}
Let $1<p<2$, $0<s<1$, $1<sp<d$, $\al,\be,\al+\be\in(-1,sp),\,1+\al+\be\neq sp$ and $q=\frac{dp}{d-sp}$. Moreover, we assume that the parameters $\al,\be$ satisfy the additional condition

\begin{equation}\label{condition}
A_{\al,\be,p}:=\inf_{\tau>0}\frac{\tau^{p\al}+\tau^{p\be}}{\tau^{\al+\be}}>0.    
\end{equation} 

Then the following weighted fractional Hardy--Sobolev--Maz'ya inequality,
\begin{align}\label{hsm1<p<2Rd+}
\nonumber
 \int_{\R^{d}_{+}}\int_{\R^{d}_{+}}\frac{|u(x)-u(y)|^{p}}{|x-y|^{d+sp}}&\,x_{d}^{\alpha}\,y_{d}^{\beta}\,dy\,dx-\mathcal{D}\int_{\R^{d}_{+}}\frac{|u(x)|^{p}}{x_{d}^{sp-\alpha-\beta}}\,dx\\
 &\geq C A_{\al,\be,p} \left(\int_{\R^d_+}|u(x)|^{q}x_{d}^{\frac{q}{p}(\al+\be)}\,dx\right)^{\frac{p}{q}},\quad u\in C_c^1(\R^d_+),
\end{align}
holds.  The constant $C$ depends on $s,p$ and $d$ and $\mathcal{D}$ is given by \eqref{D}.
\end{thm}

\begin{rem}
 The condition \eqref{condition} is fulfilled for example when the parameters $\al,\be$ satisfy $((p-1)\al-\be)((p-1)\be-\al)\leq 0.$ In this particular case, the constant $CA_{\al,\be,p}$ in \eqref{hsm1<p<2Rd+} does not depend on $\al$ and $\be$. Indeed, one has $A_{\al,\be,p}\geq 1$.
\end{rem}

\begin{thm}[\textbf{Fractional Hardy--Sobolev--Maz'ya inequality on general domains for $\boldsymbol{1<p<2}$}]\label{twHSMgeneraldomains}
  Let $1<p<2$, $0<s<1$ and $1<sp<d$. Then there is a constant $\sigma_{d,s,p}>0$ such that
\begin{equation}\label{hardysobolevmazyageneraldomains}
\int_{\Omega}\int_{\Omega}\frac{|u(x)-u(y)|^p}{|x-y|^{d+sp}}\,dy\,dx-\mathcal{D}\int_{\Omega}\frac{|u(x)|^p}{m_{sp}(x)^{sp}}\,dx\geq\sigma_{d,s,p}\left(\int_{\Omega}|u(x)|^q \,dx\right)^{\frac{p}{q}},    
\end{equation}
for all open and proper $\Omega\subset\R^d$ and all $u\in C_c^{1}(\Omega)$. Here $q=\frac{dp}{d-sp}$, $\mathcal{D}=\mathcal{D}_{d,s,p}$ and $m_\rho$ denotes the pseudodistance
$$
m_{\rho}(x)^{\rho}=\frac{2\pi^{\frac{d}{2}}\Gamma\left(\frac{1+\rho}{2}\right)}{\Gamma\left(\frac{d+\rho}{2}\right)}\left(\int_{\mathbb{S}^{d-1}}\frac{d\omega}{d_{\omega}(x)^{\rho}}\right)^{-1},
$$
where $d_{\omega}(x)=\inf\{|t|:x+t\omega\notin\Omega\}$ for $x\in\R^d$ and $\omega\in\mathbb{S}^{d-1}$. 
\end{thm}
For convex $\Omega$ we have $m_\rho (x)\leq\dist(x,\partial\Omega)$ (see \cite{LossSloane}), hence, \eqref{hardysobolevmazyageneraldomains} can be rewritten using the standard distance to the boundary.

\begin{rem}
It is not clear if the (unweighted or weighted) fractional Hardy--Sobolev--Maz'ya inequality on the half-space and general unbounded domains is valid for the range of parameters $0<sp<1$. Our proofs of Theorems \ref{tw12} and \ref{twHSMgeneraldomains} does not resolve this case and we have not been able to find in the literature any information on this situation. Interestingly, according to our best knowledge, it is also not known if the local version of the Hardy--Sobolev--Maz'ya inequality is satisfied, when $1<p<2$. 
\end{rem}
\begin{rem}
We believe that it is also possible to obtain a weighted form of the inequality \eqref{hardysobolevmazyageneraldomains}. To this end, one should properly modify all the technical results from \cite{MR2910984}. However, such weighted Hardy--Sobolev--Maz'ya inequalities were not investigated even for the case $p\geq 2$.
\end{rem}

\noindent
\textbf{Acknowledgement.} The authors would like to thank the anonymous referee for numerous comments, which led to an improvement of the manuscript.

\section{Constants}
The two important constants that we use throughout the paper are given by
\begin{equation}\label{C}
\mathcal{C}=\mathcal{C}(d,s,p,\alpha,\beta)=\int_{0}^{1}r^{sp-1}\left(r^{\alpha}+r^{\beta}\right)\left|1-r^{\frac{d-\alpha-\beta-sp}{p}}\right|^{p}\Phi_{d,s,p}(r)\,dr.
\end{equation}
where $\al,\be,\al+\be\in(-sp,d)$, the function $\Phi_{d,s,p}$ is defined by the formula
\begin{equation}
\Phi_{d,s,p}(r)=\begin{cases}
\left|\mathbb{S}^{d-2}\right|\displaystyle\int_{-1}^{1}\frac{\left(1-t^{2}\right)^{\frac{d-3}{2}}}{\left(1-2tr+ r^2
  \right)^{\frac{d+sp}{2}}}\,dt,\,d\geq 2 \\
(1-r)^{-1-sp}+(1+r)^{-1-sp},\,d=1.
\end{cases}
\end{equation}
and
\begin{equation}\label{D}
\mathcal{D}=\mathcal{D}(d,s,p,\alpha,\beta)=\frac{\pi^{\frac{d-1}{2}}\Gamma\left(\frac{1+sp}{2}\right)}{\Gamma\left(\frac{d+sp}{2}\right)}\int_{0}^{1}\frac{\left(t^{\alpha}+t^{\beta}\right)\left|1-t^{-\frac{1+\alpha+\beta-sp}{p}}\right|^{p}}{(1-t)^{1+sp}}\,dt,
\end{equation}
where $\al,\be,\al+\be\in(-1,sp)$. We will also use the notation $\mathcal{D}_{d,s,p}:=\mathcal{D}(d,s,p,0,0)$ For $p=2$ the constants above reduce to expressions involving only Beta and Gamma functions, see \cite[Propositions 8 and 10]{sharpweighted}. As it was shown in \cite{sharpweighted}, these constants are the largest for which the left hand side of \eqref{HardyRd} or \eqref{HardyHS} stays nonnegative for all compactly supported functions~$u$.

\section{Proofs}

To prove  Proposition~\ref{remainder1<p<2}, we need a purely analytical result, being an analogue of the inequality \cite[(2.16)]{MR2469027}.
\begin{prop}\label{proposition}
Let $a\in\R$, $0\leq t\leq 1$ and $1<p<2$. Then the following inequality holds,
\begin{equation}\label{prop81}
|a-t|^{p}\geq (1-t)^{p-1}\left(|a|^p-t\right)+C_p t\left(a^{\langle p/2 \rangle}-1\right)^2,
\end{equation}
where 
\begin{equation}\label{Cp}
C_p=\max\left\{\frac{p-1}{p},\frac{p(p-1)}{2}\right\}.
\end{equation}
Moreover, for $a\geq 0$ it holds
\begin{equation}\label{prop81pos}
|a-t|^{p}\geq (1-t)^{p-1}\left(|a|^p-t\right)+(p-1)t\left(a^{p/2}-1\right)^2. 
\end{equation}
\end{prop}
 
\begin{rem}
  After finishing present article we learned that  \cite[Lemma 3.8]{MR4597627} contains a~result
  similar to our Proposition~\ref{proposition}, with different right-hand side and stronger in the sense that it gives double-sided
  inequality with different constants. However, in that result the right-hand side of \eqref{prop81}
  is of more complicated form not readily suitable for our applications, and it is also not
  straightforward to prove  Proposition~\ref{proposition} using  \cite[Lemma 3.8]{MR4597627}.
\end{rem}

\begin{rem}
The constant $C_p$ in the inequality \eqref{prop81} is very likely not sharp. However, it seems difficult  to find explicitly the best constant in this setting, as numerical computations show. The constant $p-1$ in \eqref{prop81pos} is optimal, which follows from its proof (consider  $a\rightarrow\infty,\,t\rightarrow 0^+$).
\end{rem}
An elementary, but rather technical proof is given in the Appendix.

\begin{proof}[Proof of Proposition~\ref{remainder1<p<2}]
We follow \cite[Proof of Proposition 2.2]{MR2469027}, with $t=w(y)/w(x)$ (assuming without loss of generality that $t\leq 1$) and $a=v(x)/v(y)$ and use Proposition \ref{proposition}. The property $\sgn(a/b)=\sgn(a)/\sgn(b)$ is used here.  To prove \eqref{generalhardyremainder1<p<23}, we also use the triangle inequality $|a-t|\geq ||a|-t|$, valid for all $a\in\mathbb{C}$.
\end{proof}

\begin{proof}[Proof of Theorems \ref{tw10} and \ref{tw11}]
This is a direct consequence of the general inequalities \eqref{generalhardyremainder1<p<21} and \eqref{generalhardyremainder1<p<22},\cite[Lemmas 8 and 10]{sharpweighted} and previous results.  Let us now briefly discuss the details of the proof.  In \cite[Lemma 8]{sharpweighted} we proved that if $\al,\be,\al+\be\in(-sp,d)$ and $w_1(x)=|x|^{-(d-\al-\be-sp)/p}$, then 
\begin{align*}
V_1(x)&:=2\lim_{\varepsilon\rightarrow 0}\int_{||x|-|y||>\varepsilon}\left(w_1(x)-w_1(y)\right)|w_1(x)-w_1(y)|^{p-2}k(x,y)\,dy\\
&=\frac{\mathcal{C}(d,s,p,\alpha,\beta)}{|x|^{sp+\alpha+\beta}}w_1(x)^{p-1}
\end{align*}
uniformly on compact sets contained in $\R^d\setminus\{0\}$. Using  Proposition \ref{remainder1<p<2} we obtain the result of Theorem \ref{tw10}. For the case of the halfspace, the proof follows the same pattern. By \cite[Lemma 10]{sharpweighted}, for $\al,\be,\al+\be\in(-1,sp)$, $\al+\be+sp\neq 1$ and $w_2(x)=x_d^{-(1+\al+\be-sp)/p}$ we have
\begin{align*}
V_2(x)&:=2\lim_{\varepsilon\rightarrow 0}\int_{|x_{d}-y_{d}|>\varepsilon,\,y_{d}>0}\left(w_2(x)-w_2(y)\right)|w_2(x)-w_2(y)|^{p-2}k(x,y)\,dy\\
&=\frac{\mathcal{D}(d,s,p,\alpha,\beta)}{x_{d}^{sp-\alpha-\beta}}w_2(x)^{p-1},
\end{align*}
uniformly on compact sets contained in $\R^d_{+}$. Hence, Proposition \ref{remainder1<p<2} again implies Theorem~\ref{tw11}.

\end{proof}
\begin{proof}[Proof of Theorem \ref{tw12}]
  The proof is analogous to the proof of the fractional Hardy--Sobolev--Maz'ya inequality for $p\geq2$ from \cite[Proof of Theorem 3]{sharpweighted}, however, there are some minor changes. We present the proof for real-valued functions; the calculations for complex-valued functions are similar. The starting point in proving the statement is of course Theorem \ref{tw11}. We keep  the notation from Theorem \ref{tw11} and put $w(x)=x_d^{-(1+\al+\be-sp)/p}$.
  We have by symmetry of $V(x,y)=\min\{w(x),w(y)\}\max\{w(x),w(y)\}^{p-1}$ that the remainder in the weighted fractional Hardy inequality satisfies
\begin{align*}
\widetilde{E}_{w}[u]&=\int_{\R^d_+}\int_{\R^d_+}\frac{\left(v(x)^{\langle p/2 \rangle}-v(y)^{\langle p/2 \rangle}\right)^{2}}{|x-y|^{d+sp}}V(x,y)x_d^{\al}y_d^{\be}\,dy\,dx\\
&=\frac{1}{2}\int_{\R^d_+}\int_{\R^d_+}\frac{\left(v(x)^{\langle p/2 \rangle}-v(y)^{\langle p/2 \rangle}\right)^{2}}{|x-y|^{d+sp}}V(x,y)\left(x_d^{\al}y_d^{\be}+x_d^{\be}y_d^{\al}\right)\,dy\,dx\\
&\geq\frac{1}{2}\iint_{\{w(x)\leq w(y)\}}\frac{\left(v(x)^{\langle p/2 \rangle}-v(y)^{\langle p/2 \rangle}\right)^{2}}{|x-y|^{d+sp}}\\
&\times x_d^{^{\frac{sp-1}{p}}}y_d^{\frac{(p-1)(sp-1)}{p}}\left[\left(\frac{x_d}{y_d}\right)^{\frac{(p-1)\al-\be}{p}}+ \left(\frac{x_d}{y_d}\right)^{ \frac{(p-1)\be-\al}{p}}\right]\,dy\,dx.
\end{align*}
We estimate in the same way as in \cite{MR2910984},
\begin{align*}
  x_d^{^{\frac{sp-1}{p}}}y_d^{\frac{(p-1)(sp-1)}{p}}&\geq
  \min\{x_d,y_d\}^{sp-1} =
  (sp-1)\int_{0}^{\infty}\ind_{(t,\infty)}(x_d)\ind_{(t,\infty)}(y_d)t^{sp-2}\,dt.
\end{align*}
By making use of the assumption \eqref{condition} with $\tau=(x_d/y_d)^{1/p}$, we conclude that
$\left(x_d/y_d\right)^{\frac{(p-1)\al-\be}{p}}+ \left(x_d/y_d\right)^{ \frac{(p-1)\be-\al}{p}}\geq A_{\al,\be,p}>0$.
Therefore, by symmetry of the integrand, $\widetilde{E}_{w}[u]$ dominates
\begin{align*}
&\frac{A_{\al,\be,p}}{2} \int_{0}^{\infty}t^{sp-2}\,dt\iint_{\{w(x)\leq w(y)\}}\frac{\left(v(x)^{\langle p/2 \rangle}-v(y)^{\langle p/2 \rangle}\right)^{2}}{|x-y|^{d+sp}}\ind_{(t,\infty)}(x_d)\ind_{(t,\infty)}(y_d)\,dy\,dx\\
&=\frac{A_{\al,\be,p}}{4}\int_{0}^{\infty}t^{sp-2}\,dt\int_{\R^d}\int_{\R^d}\frac{\left(v(x)^{\langle p/2 \rangle}-v(y)^{\langle p/2 \rangle}\right)^{2}}{|x-y|^{d+sp}}\ind_{(t,\infty)}(x_d)\ind_{(t,\infty)}(y_d)\,dy\,dx\\
&=\frac{A_{\al,\be,p}}{4}\int_{0}^{\infty}t^{sp-2}\,dt\int_{\{x_d>t\}}\int_{\{y_d>t\}}\frac{\left(v(x)^{\langle p/2 \rangle}-v(y)^{\langle p/2 \rangle}\right)^{2}}{|x-y|^{d+sp}}\,dy\,dx.
\end{align*}

By the critical Sobolev inequality \cite[Lemma 2.1]{MR2910984}, see also \cite[Corollary 2]{MR2607491}, with the parameters $s'=\frac{sp}{2}$, $p'=2$ and $q'=\frac{dp'}{d-s'p'}=\frac{2d}{d-sp}$ the above dominates (up to a constant depending only on $d$, $s$ and $p$)
\begin{align*}
\int_{0}^{\infty}\left(\int_{\{x_d>t\}}|v(x)^{\langle p/2\rangle}|^{q'}\,dx\right)^{\frac{p'}{q'}}t^{sp-2}\,dt=\int_{0}^{\infty}\left(\int_{\{x_d>t\}}|v(x)|^q\,dx\right)^{\frac{q}{p}}t^{sp-2}\,dt.
\end{align*}
Recalling the definition of $v$ and Minkowski's inequality yield the result.
\end{proof}

In order to prove the fractional Hardy--Sobolev--Maz'ya inequality for general domains, we first need to generalize technical Lemmas from \cite{MR2910984} to our case of $1<p<2$.

\begin{lem}\label{lem8}
Let $0<s<1$ and $1<p<2$ with $sp>1$. Then, for all $f\in C^1((0,1))$ with $f(0)=0$ it holds
\begin{align}\label{int01}
&\int_{0}^{1}\int_{0}^{1}\frac{|f(x)-f(y)|^p}{|x-y|^{1+sp}}\,dy\,dx-\mathcal{D}_{1,s,p}\int_{0}^{1}\frac{|f(x)|^p}{x^{sp}}\,dx\\
\nonumber 
&\geq C_p\int_{0}^{1}\int_{0}^{1}\frac{\left(v(x)^{\langle p/2\rangle}-v(y)^{\langle p/2\rangle}\right)^2}{|x-y|^{1+sp}}U(x,y)\,dy\,dx+\int_{0}^{1}W_{p,s}(x)|v(x)|^p w(x)^p\,dx,
\end{align}
where $C_p$ is given by \eqref{Cp}, $w(x)=x^{\frac{sp-1}{p}}$, $v=f/w$ and
\[
U(x,y)=\min\{w(x),w(y)\}\max\{w(x),w(y)\}^{p-1}.
\]
The function $W_{p,s}$ is like in \cite[Lemma 4.3]{MR2910984}, i.e.,
it is bounded away from zero and satisfies
\[
W_{p,s}(x) \approx x^{-(p-1)(ps-1)/p} \qquad\text{for }  x\rightarrow 0^+
\]
and
\[
W_{p,s}(x) \approx
\begin{cases}
 1 & \text{if}\ p-1-ps>0 \,, \\
|\log(1-x)| & \text{if} \ p-1-ps=0 \,, \\
(1-x)^{-1-ps+p}  & \text{if} \ p-1-ps<0 \,,
\end{cases}
\qquad \text{for } x\rightarrow 1^-.
\]
\end{lem}
\begin{proof}
By the ground state representation \eqref{generalhardyremainder1<p<21},
\begin{align*}
\int_0^1 \int_0^1 \frac{|f(x)-f(y)|^p}{|x-y|^{1+ps} } \,dx\,dy
&\geq \int_0^1 V(x) |f(x)|^p\,dx\\
&+  C_p\int_{0}^{1}\int_{0}^{1}\frac{\left(v(x)^{\langle p/2\rangle}-v(y)^{\langle p/2\rangle}\right)^2}{|x-y|^{1+sp}}U(x,y)\,dy\,dx
\end{align*}
with
\[
V(x) := 2 \omega(x)^{-p+1} \int_0^1 \left(\omega(x) -\omega(y)\right)
\left|\omega(x) - \omega(y) \right|^{p-2} |x-y|^{-1-ps}\,dy
\]
(understood as principal value integral).
It suffices to prove that $V(x) \geq \mathcal{D}_{1,s,p} x^{-sp} + W_{p,s}(x)$, with $W_{p,s}$ defined above. This can be done
in exactly the same way as in the proof of \cite[Lemma 4.3]{MR2910984}, as it turns out that it is enough to assume $p>1$ in this part of the proof. We omit the details.
\end{proof}
The following result is an analogue of \cite[Corollary 4.4]{MR2910984}.
\begin{cor}\label{cor9}
Let $0<s<1$ and $1<p<2$ with $sp>1$. Then, for all $f$ with $f(-1)=f(1)=0$ it holds
\begin{align}
&\int_{-1}^{1}\int_{-1}^{1}\frac{|f(x)-f(y)|^p}{|x-y|^{1+sp}}\,dy\,dx-\mathcal{D}_{1,s,p}\int_{-1}^{1}\frac{|f(x)|^p}{(1-|x|)^{sp}}\,dx \label{eq:cor-w}\\
\nonumber 
&\geq C_p\left(\int_{-1}^{0}\int_{-1}^{0}+\int_{0}^{1}\int_{0}^{1}\right)\frac{\left(v(x)^{\langle p/2\rangle}-v(y)^{\langle p/2\rangle}\right)^2}{|x-y|^{1+sp}}U(x,y)\,dy\,dx+c_{p,s}\int_{-1}^{1}|v(x)|^p w(x)\,dx, \nonumber
\end{align}
where $w(x)=\left(1-|x|\right)^{\frac{sp-1}{p}}$, $v=f/w$ and $U(x,y)=\min\{w(x),w(y)\}\max\{w(x),w(y)\}^{p-1}$.
\end{cor}
\begin{proof}

By adding the inequalities \eqref{int01} for $f_1(x)=f(1-x)$ and $f_2(x)=f(x-1)$ and changing the variables, we obtain that the left hand side of \eqref{eq:cor-w} is bounded from below
by
\begin{align*}
C_p\left(\int_{-1}^{0}\int_{-1}^{0}+\int_{0}^{1}\int_{0}^{1}\right)&\frac{\left(v(x)^{\langle p/2\rangle}-v(y)^{\langle p/2\rangle}\right)^2}{|x-y|^{1+sp}}U(x,y)\,dy\,dx \\
&+
\int_{-1}^{1} W_{p,s}(1-|x|) |v(x)|^p (1-|x|)^{sp-1}\,dx .
\end{align*}
Since $p-1-ps < 0$, we have that
\[
W_{p,s}(1-|x|) \geq c_{p,s} (1-|x|)^{-(p-1)(ps-1)/p},\quad\textrm{for $x\in (-1,1)$;}
\]
note that this bound is not optimal, as $W_{p,s}(1-|x|)$ is infinite at zero.
Therefore
\[
W_{p,s}(1-|x|)(1-|x|)^{sp-1} \geq  c_{p,s} (1-|x|)^{\frac{sp-1}{p}}
\]
for $x\in (-1,1)$ and the proof is finished.
\end{proof}

Finally, we are in a~position to prove a version of \cite[Lemma 4.1 and Corollary 4.2]{MR2910984} for $1<p<2$.
\begin{lem}\label{lem10}
Let $0<s<1$, $q\geq 1$ and $1<p<2$ with $sp>1$. Then there exists a constant $c=c(p,q,s)>0$ such that for all $f\in C_c^{1}((-1,1))$ it holds
\begin{equation}\label{normsup1}
\|f\|_{\infty}^{p+q(sp-1)}\leq c\left(\int_{-1}^{1}\int_{-1}^{1}\frac{|f(x)-f(y)|^p}{|x-y|^{1+sp}}\,dy\,dx-\mathcal{D}_{1,s,p}\int_{-1}^{1}\frac{|f(x)|^p}{(1-|x|)^{sp}}\,dx\right)\|f\|_{q}^{q(sp-1)}. 
\end{equation}
In consequence, for all open proper $\Omega\subset\R$ and all $f\in C_c^{1}(\Omega)$,
\begin{equation}\label{normsup2}
\|f\|_{\infty}^{p+q(sp-1)}\leq c\left(\int_{\Omega}\int_{\Omega}\frac{|f(x)-f(y)|^p}{|x-y|^{1+sp}}\,dy\,dx-\mathcal{D}_{1,s,p}\int_{\Omega}\frac{|f(x)|^p}{\dist(x,\partial\Omega)^{sp}}\,dx\right)\|f\|_{q}^{q(sp-1)}. \end{equation}
\end{lem}
\begin{proof}
The proof of $\eqref{normsup1}$ is a slight modification of \cite[Proof of Lemma 4.1]{MR2910984}. Let $w(x)=(1-|x|)^{(sp-1)/p}$ and $x=f/w$. By Corollary \ref{cor9} it suffices to prove the equivalent inequality
\begin{align*}
\|vw\|^{p+q(sp-1)}&\leq c\Bigg(\int_{-1}^{1}|v(x)|^p w(x)\,dx\\
&+\left(\int_{-1}^{0}\int_{-1}^{0}+\int_{0}^{1}\int_{0}^{1}\right)\frac{\left(v(x)^{\langle p/2\rangle}-v(y)^{\langle p/2\rangle}\right)^2}{|x-y|^{1+sp}}U(x,y)\,dy\,dx\Bigg)\|vw\|^{q(sp-1)}.
\end{align*}
Without loss of generality, we may assume that $v$ is nonnegative and that for some $x_0\in[0,1)$ it holds $v(x_0)w(x_0)=\|vw\|_{\infty}<\infty$. We distinguish three cases: \\
Case 1: $x_0\in[0,\tfrac{1}{2}]$ and $vw\geq c_1v(x_0)w(x_0)$ on $[0,\tfrac{1}{2}]$;\\
Case 2: $x_0\in[0,\tfrac{1}{2}]$ and there is a $z\in[0,\frac{1}{2}]$ such that $v(z)w(z)\leq c_1 v(x_0)w(x_0)$;\\
Case 3: $x_0\in(\tfrac{1}{2},1)$.

The only things that essentially change, compared to \cite[Proof of Lemma 4.1]{MR2910984}, are estimates of the integral $$\int_{0}^{1}\int_{0}^{1}\frac{|v(x)-v(y)|^p}{|x-y|^{1+sp}}w(x)^{p/2}w(y)^{p/2}\,dy\,dx$$ in the Case $2$ and $3$, which is replaced by $$\int_{0}^{1}\int_{0}^{1}\frac{\left(v(x)^{\langle p/2\rangle }-v(y)^{\langle p/2\rangle }\right)^2}{|x-y|^{1+sp}}U(x,y)\,dy\,dx$$ in our setting. 
For example, let us look at the case $2$. Let z be closest possible to $x_0$, so that $v(z)w(z)=c_1 v(x_0)w(x_0)$
and $vw\geq c_1 v(x_0)w(x_0)$ on the interval $I$ with endpoints $x_0$ and $z$.
We may write 
\begin{align*}
&\int_{0}^{1}\int_{0}^{1}\frac{\left(v(x)^{\langle p/2\rangle }-v(y)^{\langle p/2\rangle }\right)^2}{|x-y|^{1+sp}}U(x,y)\,dy\,dx\\
&=2\iint_{\{w(x)<w(y)\}}\frac{\left(v(x)^{\langle p/2\rangle }-v(y)^{\langle p/2\rangle }\right)^2}{|x-y|^{1+sp}}w(x)w(y)^{p-1}\,dy\,dx\\
&\geq 2\iint_{(I\times I)\cap \{w(x)<w(y)\}}\frac{\left(v(x)^{\langle p/2\rangle }-v(y)^{\langle p/2\rangle }\right)^2}{|x-y|^{1+sp}}w(x)w(y)^{p-1}\,dy\,dx\\
&\geq 2w(\tfrac{1}{2})^p\iint_{(I\times I)\cap \{w(x)<w(y)\}}\frac{\left(v(x)^{\langle p/2\rangle }-v(y)^{\langle p/2\rangle }\right)^2}{|x-y|^{1+sp}}\,dy\,dx\\
&=w(\tfrac{1}{2})^p\int_{I}\int_I\frac{\left(v(x)^{\langle p/2\rangle }-v(y)^{\langle p/2\rangle }\right)^2}{|x-y|^{1+sp}}\,dy\,dx.
\end{align*}
Using  $v(x_0)\geq c v(z)$ for some constant $c$ and \cite[Lemma 4.5]{MR2910984} with parameters $p'=2$ and $s'=sp/2$, the double integral above dominates a constant times 
 \begin{align*}
\left(v(x_0)^{\langle p/2 \rangle }-v(z)^{\langle p/2\rangle }\right)^{2}|x_0-z|^{1-sp}&\geq c'|v(x_0)^{\langle p/2\rangle}|^2|x_0-z|^{1-sp}\\
&\geq c''|v(x_0)w(x_0)|^p |x_0-z|^{1-sp}
\end{align*}
 which gives the desired bound from below. The case $3$ is treated similarly as above and in \cite{MR2910984} and we omit it. The rest of the proof remains unchanged, hence, \eqref{normsup1} follows. The inequality \eqref{normsup2} can then be derived from \eqref{normsup1} by a standard argument of translation, dilation and decomposition of any open set in $\R$ into a countable sum of intervals.
\end{proof}
\begin{proof}[Proof of Theorem \ref{twHSMgeneraldomains}]
By making use of Lemmas \ref{lem8} and \ref{lem10}, the proof is a copy of \cite[Proof of Theorem 1.1]{MR2910984}
\end{proof}

\section{Appendix}\label{Appendix}
\begin{proof}[Proof of Proposition \ref{proposition}]
We denote
$$
f(a,t)=\frac{|a-t|^p-(1-t)^{p-1}(|a|^p-t)}{t|a^{\langle p/2 \rangle}-1|^2}.
$$
Suppose first that $a>1>t$. Then, one may check that the partial derivative of $f$ with respect to $a$ is given by the formula
$$
\frac{\partial f }{\partial a}(a,t)=p\frac{(1-t)^{p-1}(a^{p-1}-ta^{ p/2 -1})-(a-t)^{p-1}(1-ta^{p/2 -1})}{t(a^{p/2}-1)^3}.
$$
To show that $\frac{\partial f}{\partial a}\leq 0$, equivalently we  will show that
$$
g(a,t):=(p-1)\log(1-t)+\log\left(a^{p-1}-ta^{p/2-1}\right)-(p-1)\log(a-t)-\log\left(1-ta^{p/2-1}\right)\leq 0.
$$
After elementary but tedious calculations we find that the partial derivative of $g$ with respect to $a$ is
\begin{align*}
\frac{t\left[\left(-pt+t-p+1\right)a^{p-2}+\left(\frac{p}{2}-1\right)ta^{p/2-2}+\frac{pt}{2}a^{3p/2-3}+\frac{p}{2}a^{p/2-1}+\left(\frac{p}{2}-1\right)a^{3p/2-2}\right]}{(a^{p-1}-ta^{p/2-1})(a-t)(1-ta^{p/2-1})}.    
\end{align*}
Notice that the denominator above is always nonnegative, as we have $ta^{p/2-1}<1$, $a^{p-1}-ta^{p/2-1}=a^{p/2-1}(a^{p/2}-t)>0$. Denote 
$$
h(a,t):=\left(-pt+t-p+1\right)a^{p-2}+\left(\frac{p}{2}-1\right)ta^{p/2-2}+\frac{pt}{2}a^{3p/2-3}+\frac{p}{2}a^{p/2-1}+\left(\frac{p}{2}-1\right)a^{3p/2-2}.
$$
It is a linear function with respect to $t$, hence, to show that $h(a,t)\leq 0$ we only need to check if $h(a,0)\leq 0$ and $h(a,1)\leq 0$ for all $a>1$. We have
\begin{align*}
F(a)&:=h(a,0)=a^{p-2}(1-p)+\frac{1}{2}pa^{p/2-1}+a^{3p/2-2}\left(\frac{p}{2}-1\right)\\
&=a^{p/2-1}\left[\frac{p}{2}-\left(1-\frac{p}{2}\right)a^{p-1}-(p-1)a^{p/2-1}\right].
\end{align*}
Standard calculation shows that the function in the square brackets above is nonpositive and vanishes at its maximum attained at $a=1$, thus, $F(a)\leq 0$, therefore we have $h(a,0)\leq 0$ for all $a>1$. Next, denote
\begin{align*}
G(a)&:=h(a,1)=2(1-p)a^{p-2}+\left(\frac{p}{2}-1\right)a^{p/2-2}+\frac{p}{2}a^{3p/2-3}+\frac{p}{2}a^{p/2-1}+\left(\frac{p}{2}-1\right)a^{3p/2-2}\\
&=a^{p/2-2}\left[2(1-p)a^{p/2}+\frac{p}{2}-1+\frac{p}{2}a^{p-1}+\frac{p}{2}a+\left(\frac{p}{2}-1\right)a^p\right]=:a^{p/2-2}H(a).
\end{align*}
We have
\begin{align*}
H'(a)&=p\left[(1-p)a^{p/2-1}+\frac{1}{2}(p-1)a^{p-2}+\frac{1}{2}+\left(\frac{p}{2}-1\right)a^{p-1}\right]
\end{align*}
and it is easy to see that $H'(a)\leq H'(1)=0$ for all $a\geq 1$. Hence, $H(a)$ is decreasing, therefore $H(a)\leq H(1)=0$. In consequence, $G(a)\leq 0$, thus $h(a,t)\leq 0$ for all $a\geq 1$ and $0\leq t\leq 1$. This means that $g(\cdot,t)$ is decreasing, hence, $g(a,t)\leq g(1,t)=0$. Overall, we get that $\frac{\partial f}{\partial a}(a,t)\leq 0$ for all $a\geq 1$. Therefore, $f(\cdot,t)$ is decreasing, thus, for $a>1$,
$$
f(a,t)\geq \lim_{a\rightarrow\infty}f(a,t)=\frac{1-(1-t)^{p-1}}{t}\geq p-1.
$$
The latter follows from concavity of the function $\tau\mapsto\tau^{p-1}$.

We now assume that $0<a<1$. Then, by the first part of the proof, 
\begin{align*}
p-1&\leq f\left(\frac{1}{a},t\right)=\frac{(1-at)^p-(1-t)^{p-1}(1-a^pt)}{t(1-a^{p/2})^2}    
\end{align*}
and we will show that the latter expression does not exceed $f(a,t)$. Equivalently, we need to show that for $0<a,t<1$ we have \begin{align*}
F(a,t):&=|a-t|^p-(1-t)^{p-1}(a^p-t)-((1-at)^p-(1-t)^{p-1}(1-a^pt))\\
&=|a-t|^p-(1-at)^p+(1+t)(1-t)^{p-1}(1-a^p)\geq 0.
\end{align*}
For $a>t$ we have
\begin{align*}
\frac{\partial F}{\partial a}&=p\left[(a-t)^{p-1}+t(1-at)^{p-1}-a^{p-1}(1+t)(1-t)^{p-1}\right]\\
&=pa^{p-1}\left[\left(1-\frac{t}{a}\right)^{p-1}+t\left(\frac{1}{a}-t\right)^{p-1}-(1+t)(1-t)^{p-1}\right]
\end{align*}
and, by a standard argument, the function in the square brackets attains its maximum equal to zero at $a=1$. For $0<a\leq t$ it holds
\begin{align*}
\frac{\partial F}{\partial a}&=p\left[-(t-a)^{p-1}+t(1-at)^{p-1}-a^{p-1}(1+t)(1-t)^{p-1}\right]\\
&=pa^{p-1}\left[-\left(\frac{t}{a}-1\right)^{p-1}+t\left(\frac{1}{a}-t\right)^{p-1}-(1+t)(1-t)^{p-1}\right]\\
&=:pa^{p-1}G(x,t),
\end{align*}
where we substitute $x=1/a\geq 1/t>1$. We have
\begin{align*}
\frac{\partial G}{\partial x}=(p-1)t\left[-(tx-1)^{p-2}+(x-t)^{p-2}\right]\leq 0,    
\end{align*}
because $x-t>tx-1$ and $p-2\leq 0$. Hence, 
$$G(x,t)\leq G\left(\frac{1}{t},t\right)=(1+t)(1-t)^{p-1}\left[\left(1+\frac{1}{t}\right)^{p-2}-1\right]\leq 0.$$
Combining all the results above we obtain that for $0<a,t<1,\,a\neq t$ it holds $\frac{\partial F}{\partial a}\leq 0$. Therefore, $F(a,t)\geq F(1,t)=0$ and that part of the proof is finished.

It suffices to show the initial inequality for $a\leq 0$ and $t>0$. Substituting $x=-a\geq 0$, we need to show that 
$$
f(x,t):=f(-a,t)=\frac{(x+t)^p-(1-t)^{p-1}(x^p-t)}{t(x^{p/2}+1)^{2}}\geq \frac{p(p-1)}{2},\,x\geq 0.
$$
Let us see that the constant $C_p$ from \eqref{Cp} satisfies $C_p=(p-1)/p$, for $1<p<\sqrt{2}$ and $C_p=p(p-1)/2$, when $\sqrt{2}\leq p <2$. Thus, in the last part of the proof we will distinguish these two cases of the range of $p$. We first assume that $p\in\left[\sqrt{2},2\right)$. By Taylor's formula we have
$$
(x+t)^p\geq x^p+ptx^{p-1}+\frac{p(p-1)}{2}t^2(x+t)^{p-2},
$$
therefore
\begin{align*}
(x+t)^p-(1-t)^{p-1}(x^p-t)&\geq x^p(1-(1-t)^{p-1})+ptx^{p-1}+t\left[(1-t)^{p-1}+\frac{p(p-1)}{2}t(x+t)^{p-2}\right]\\
&\geq x^p(1-(1-t)^{p-1})+ptx^{p-1}+t\left[1-t^{p-1}+\frac{p(p-1)}{2}t(x+1)^{p-2}\right]\\
&\geq (p-1)tx^p+ptx^{p-1}+\frac{p(p-1)}{2}t(x+1)^{p-2},
\end{align*}
as it is straightforward to see that the function in the square brackets above attains its minimum at $t=1$. Hence,
\begin{align*}
f(x,t)&\geq\frac{(p-1)x^p+px^{p-1}+\frac{p(p-1)}{2}(1+x)^{p-2}}{\left(x^{p/2}+1\right)^2}
\end{align*}
and we will show that the latter expression is bounded from below by $\frac{p(p-1)}{2}$. Equivalently, defining
$$
g(x):=\frac{1}{p(p-1)}\left((p-1)x^p+px^{p-1}+\frac{p(p-1)}{2}(1+x)^{p-2}-\frac{p(p-1)}{2}\left(x^{p/2}+1\right)^2\right),
$$
one has to check that $g(x)\geq 0$. The derivative of $g$ is given by
\begin{align}\label{g'}
\nonumber
 g'(x)&=x^{p/2-1}\left(x^{p/2-1}-\frac{p}{2}\right)+\left(1-\frac{p}{2}\right)\left(x^{p-1}-(1+x)^{p-3}\right)\\
 &\geq x^{p/2-1}\left(x^{p/2-1}-\frac{p}{2}\right)+\left(1-\frac{p}{2}\right)\left(x^{p-1}-1\right).
\end{align}
For $0<x<1$ we have $x^{p/2-1}>1$, hence
\begin{align*}
  g'(x) &\geq
x^{p/2-1} \left(x^{p/2-1}-\frac{p}{2}\right)+\left(1-\frac{p}{2}\right)\left(x^{p-1}-1\right)\\
&> 1-\frac{p}{2}+\left(1-\frac{p}{2}\right)\left(x^{p-1}-1\right)\\
&=\left(1-\frac{p}{2}\right)x^{p-1}\geq 0.
\end{align*}
For $x>1$ and $x^{p/2-1}\geq p/2$, that is $1<x\leq (\frac{2}{p})^{2/(2-p)}$ the two terms in \eqref{g'} are nonnegative, hence  to prove that $g'(x)\geq 0$, it suffices to assume that $x>(\frac{2}{p})^{2/(2-p)}$. Then we use $x^{p/2-1}< p/2$ and obtain
\begin{align*}
  g'(x)&\geq x^{p-2}-\frac{p^2}{4}+\left(1-\frac{p}{2}\right)\left(x^{p-1}-1\right)\\
  &=\left(1-\frac{p}{2}\right)x^{p-1}+x^{p-2}-\frac{1}{4}\left((p-1)^2+3\right)\\
  &\geq \left(1-\frac{p}{2}\right)x^{p-1}+x^{p-2}-\frac{p+2}{4}:=h(x),
\end{align*}
since $(p-1)^2<p-1$. The global minimum of the function $h$ is attained at the point $x_0=\frac{2}{p-1}$, therefore
\begin{align*}
 h(x)&\geq h(x_0)=2^{p-2}(p-1)^{1-p}-\frac{p+2}{4}.   
\end{align*}
To show that the expression above is nonnegative for $p\in[\sqrt{2},2]$, equivalently, one has to check that
\[
k(p):=p\log 2-(p-1)\log(p-1)-\log(p+2)\geq 0\,,\quad p\in[\sqrt{2},2).
\]

We have $k'(p)=\log 2-1-\log(p-1)-\frac{1}{p+2}$ and $k''(p)=-\frac{p^2+3p+5}{(p-1)(p+2)^2}\leq 0$, therefore $k$ is concave on $[1,2]$.
It holds $k(\sqrt{2})\approx 0.117>0$ and $k(2)=0$, so by concavity it has to be $k(p)\geq 0$ for all $p\in[\sqrt{2},2]$. In consequence, the function $g'$ is always nonnegative, therefore $g(x)\geq g(0)=0$ and that finishes this part of the proof.

 To end the proof we note that, for $p\in(1,2)$, by making use of the fact that $(x+t)^p\geq x^p+t^p$, similarly as before we may show that for $a\leq 0$ it holds
 \begin{align*}
 f(x,t)&\geq\frac{x^p+t^p-(1-t)^{p-1}(x^p-t)}{t(x^{p/2}+1)^2}\\
 &=\frac{x^p\frac{1-(1-t)^{p-1}}{t}+t^{p-1}+(1-t)^{p-1}}{(x^{p/2}+1)^2}\\
 &\geq\frac{(p-1)x^p+t+(1-t)}{(x^{p/2}+1)^2}=\frac{(p-1)x^p+1}{(x^{p/2}+1)^2}\\
 &\geq\min_{x\geq 0}\frac{(p-1)x^p+1}{\left(x^{p/2}+1\right)^2}=\frac{p-1}{p}.
 \end{align*}
 That ends the proof.
\end{proof}


\def\cprime{$'$} \def\cprime{$'$}
  \def\polhk#1{\setbox0=\hbox{#1}{\ooalign{\hidewidth
  \lower1.5ex\hbox{`}\hidewidth\crcr\unhbox0}}}
  \def\polhk#1{\setbox0=\hbox{#1}{\ooalign{\hidewidth
  \lower1.5ex\hbox{`}\hidewidth\crcr\unhbox0}}}
  \def\polhk#1{\setbox0=\hbox{#1}{\ooalign{\hidewidth
  \lower1.5ex\hbox{`}\hidewidth\crcr\unhbox0}}} \def\cprime{$'$}
  \def\cprime{$'$} \def\cprime{$'$} \def\cprime{$'$} \def\cprime{$'$}
  \def\cprime{$'$}

\end{document}